\documentclass[12pt]{amsart}
\usepackage{amssymb}

\newtheorem{thm}{Theorem}

\newtheorem{lem}{Lemma}
\newtheorem{cor}{Corollary}
\begin{document}
\author{G\'erard Endimioni}
\address{C.M.I-Universit\'{e} de Provence\\
39, rue F. Joliot-Curie, F-13453 Marseille Cedex 13}
\email{endimion@gyptis.univ-mrs.fr}
\title[Automorphisms fixing every normal subgroup]
{Automorphisms fixing every normal subgroup of a nilpotent-by-abelian group}
%
\begin{abstract} Among other things, we prove that the group of automorphisms fixing every normal subgroup of a (nilpotent of class $c$)-by-abelian group  is (nilpotent of class $\leq c$)-by-metabelian.
In particular, the group of automorphisms fixing every normal subgroup of a metabelian group is soluble of derived length at most 3. An example shows that this bound cannot be improved. 
\end{abstract}
\maketitle
%
\section{Introduction and results.}
If $G$ is a group, we write $\rm{Aut}(G)$ for the group of all automorphisms of $G$. An automorphism $\varphi\in\rm{Aut}(G)$ is said to be {\em normal}
if $\varphi(H)=H$ for each normal subgroup $H$ of $G$.
These automorphisms form a normal subgroup of $\rm{Aut}(G)$, denoted by $\rm{Aut_n}(G)$. 
Obviously, $\rm{Aut_n}(G)$ contains the subgroup $\rm{Inn}(G)$
of all inner automorphisms of $G$. 
These subgroups can coincide,
for instance when $G$ is a nonabelian free group \cite{LU}.
We have a similar result when $G$ is a nonabelian free nilpotent group of class $c=2$,  
but if $c\geq 3$, then the quotient $\rm{Aut_n}(G)/\rm{Inn}(G)$ is infinite \cite{EN}. 
Also one can find in \cite{EN1} a description of the normal automorphisms of a free metabelian nilpotent group.
In general, the structure of $\rm{Aut_n}(G)$ can be quite diverse since for an arbitrary finite group $F$, there is a finite semisimple group $G$ such that $\rm{Aut_n}(G)/\rm{Inn}(G)$ has a subgroup isomorphic with $F$ \cite{RO}. Nevertheless, as one can expect, it is possible to obtain some information about $\rm{Aut_n}(G)$ when $G$ belongs to certain classes of groups. For instance, when $G$ is nilpotent, the group $\rm{Aut_n}(G)$ is nilpotent-by-abelian \cite{FG}. In the same paper, the authors show that  $\rm{Aut_n}(G)$ is polycyclic when $G$ is. In particular, if $G$ is a finite soluble group, then so is  $\rm{Aut_n}(G)$. We do not know if this assertion remains true
without the term "finite". In other words, is $\rm{Aut_n}(G)$ soluble when $G$ is soluble ?
Here we respond to this question in the positive when  $G$ is soluble of derived length $\leq 2$, and more generally when $G$ is nilpotent-by-abelian.
\begin{thm} The group of all normal automorphisms of a (nilpotent of class $c$)-by-abelian group  is (nilpotent of class $\leq c$)-by-metabelian. 
\end{thm}
When $c=1$, we obtain:
\begin{cor} The group of all normal automorphisms of a metabelian group is soluble of derived length at most 3. 
\end{cor}
As usual, denote by $S_4$ the symmetric group of degree 4 and by $A_4$ its alternating subgroup. Then $\rm{Aut_n}(A_4)=\rm{Aut}(A_4)$ is isomorphic to $S_4$. Since $A_4$ is metabelian and $S_4$ is soluble of derived length 3, the bound of the derived length given in the corollary above cannot be improved. 

We shall see that the proof of Theorem 1 also leads to the following result.
\begin{thm} Let $G$ be a (nilpotent of class $c$)-by-abelian group. Suppose that 
its abelianization  $G/G'$ is either finite or infinite non-periodic (that is the case for example when $G$ is finitely generated).
Then the group of all normal automorphisms of $G$  is virtually (nilpotent of class $\leq c$)-by-abelian.
\end{thm}
In particular, the group of normal automorphisms of a finitely generated metabelian group is 
virtually metabelian. 

Since a supersoluble group is nilpotent-by-abelian and finitely generated, 
it follows from Theorem 2 that its group of all normal automorphisms is
virtually nilpotent-by-abelian. In fact, we have a stronger result.
\begin{thm} The group of all normal automorphisms of a supersoluble group
is finitely generated and nilpotent-by-(finite and supersoluble).
\end{thm}
%
 \section{Proofs.} 
Let $G$ be a group and let $G'$ denote its derived subgroup.
Clearly, each normal automorphism $f\in\rm{Aut_n}(G)$ induces in $G/G'$ a normal 
automorphism $f'\in\rm{Aut_n}(G/G')$. Consider the homomorphism  $\Phi\colon
\rm{Aut_n}(G)\to\rm{Aut_n}(G/G')$ defined by $\Phi(f)=f'$ and put $K=\ker\Phi$.
In other words, $K$ is the set of normal IA-automorphisms of $G$ (recall that an
automorphism of $G$ is said to be an {\em IA-automorphism} if it induces the identity automorphism in $G/G'$). Before to prove Theorem 1, we establish a preliminary result regarding the elements of $K$.
\begin{lem} Let $k$ be a positive integer. In a group $G$, consider an element $a\in\gamma_k(G')$, where 
$\gamma_k(G')$ denotes the $k$th term of the lower central series of $G'$.
If $f$ and $g$ are normal IA-automorphisms of $G$ ({\em i.e.} $f,g\in K$), we have:
\begin{enumerate}
  \item[(i)]  for all $u\in G$, $f(u^{-1}au)\equiv u^{-1}f(a)u\;\;{\rm mod}\,\gamma_{k+1}(G')$;
  \item[(ii)] $g^{-1}\circ f^{-1}\circ g\circ f(a)\equiv a\;\;{\rm mod}\,\gamma_{k+1}(G')$.
\end{enumerate}
\end{lem}
\begin{proof} (i) We have $f(u)=uu'$ for some $u'\in G'$, whence
$$f(u^{-1}au)=u'^{-1}u^{-1}f(a)uu'=u^{-1}f(a)u[u^{-1}f(a)u,u'],$$
where the commutator $[x,y]$ is defined by $[x,y]=x^{-1}y^{-1}xy$.
Since $[u^{-1}f(a)u,u']$ belongs to $\gamma_{k+1}(G')$, the result follows.

(ii) Since $f$ and $g$ preserve the normal closure of $a$ in $G$, there exist elements $u_1,\ldots,u_r,v_1,\ldots,v_s\in G$ and integers $\lambda_1,\ldots,\lambda_r,\mu_1,\ldots,\mu_s$ such that
$$f(a)=\prod_{i=1}^{r}u_i^{-1}a^{\lambda_i}u_i\:\: {\rm and}\:\:
g(a)=\prod_{i=1}^{s}v_i^{-1}a^{\mu_i}v_i.$$
Notice that in these products, the order of the factors is of no 
consequence modulo $\gamma_{k+1}(G')$. Using (i), one can then write
\begin{eqnarray*}
g\circ f(a) & \equiv & \prod_{i=1}^{r}g(u_i^{-1}a^{\lambda_i}u_i)\; \equiv\;  
\prod_{i=1}^{r}u_i^{-1}g(a)^{\lambda_i}u_i \\
{} & \equiv & \prod_{i=1}^{r}\prod_{j=1}^{s}
u_i^{-1}v_j^{-1}a^{\lambda_i\mu_j}v_ju_i \\
{} & \equiv & \prod_{i=1}^{r}\prod_{j=1}^{s}
[v_j,u_i]^{-1}v_j^{-1}u_i^{-1}a^{\lambda_i\mu_j}u_iv_j[v_j,u_i]\;\;{\rm mod}\,\gamma_{k+1}(G').
\end{eqnarray*}
Since 
$$[v_j,u_i]^{-1}v_j^{-1}u_i^{-1}a^{\lambda_i\mu_j}u_iv_j[v_j,u_i]\equiv
v_j^{-1}u_i^{-1}a^{\lambda_i\mu_j}u_iv_j  \;\;{\rm mod}\,\gamma_{k+1}(G'),$$
we obtain
$$g\circ f(a)  \equiv  \prod_{i=1}^{r}\prod_{j=1}^{s}
v_j^{-1}u_i^{-1}a^{\lambda_i\mu_j}u_iv_j
 \equiv  f\circ g(a)  \;\;{\rm mod}\,\gamma_{k+1}(G').$$
It follows $g^{-1}\circ f^{-1}\circ g\circ f(a)\equiv a\;\;{\rm mod}\,\gamma_{k+1}(G')$, which is the desired result.
\end{proof} 
\begin{proof}[Proof of Theorem 1] Thus we suppose now that $G$ is
(nilpotent of class c)-by-abelian, and so $\gamma_{c+1}(G')$ is trivial.
Recall that $K$ is the kernel of the the homomorphism  $\Phi\colon
\rm{Aut_n}(G)\to\rm{Aut_n}(G/G')$ defined above.
The group $G/G'$ being abelian, each normal automorphism of $G/G'$ is in fact a power automorphism, that is, an automorphism fixing setwise every subgroup of $G/G'$. Since the group of all power automorphisms of a group is always abelian \cite{CO}, so is $\rm{Aut_n}(G/G')$.
The group $\rm{Aut_n}(G)/K$ being isomorphic with a subgroup of $\rm{Aut_n}(G/G')$, it is
abelian, and so $\rm{Aut_n}(G)/K'$ is metabelian. It remains to see that $K'$ is nilpotent, of class at most $c$. For that, we notice that $K'$ stabilizes the normal series (of length $c+1$)
$$1=\gamma_{c+1}(G')\unlhd \gamma_{c}(G')\unlhd\cdots \gamma_{2}(G')\unlhd \gamma_{1}(G')=G'\unlhd G.$$
Indeed, the induced action of $K'$ on the factor $G/G'$ is trivial since every element of $K$ is an IA-automorphism. On the other factors, it is a consequence of the second part of the lemma above.
By a well-known result of Kalu\v{z}nin (see for instance \cite[p. 9]{SE}), it follows that $K'$ is nilpotent of class $\leq c$, as required.
\end{proof} 
\begin{proof}[Proof of Theorem 2] 
First notice that in an abelian group which is either finite or infinite non-periodic,
the group of all power automorphisms is finite. That is trivial when the group is finite. In the second case, the group of power automorphisms has order 2, the only non-identity
power automorphism being the inverse function $x\mapsto x^{-1}$
(see for instance \cite[Corollary 4.2.3]{CO} or \cite[13.4.3]{RO1}). Therefore, coming back to the proof of Theorem 1 when $G/G'$ is either finite or infinite non-periodic, we can assert that $\rm{Aut_n}(G)/K$ is finite. Since  $K$ is 
(nilpotent of class $\leq c$)-by-abelian, the result follows.
\end{proof} 
\begin{proof}[Proof of Theorem 3] 
Let $G$ be a supersoluble group. Since $G$ is polycyclic, 
so is $\rm{Aut_n}(G)$ \cite{FG}.
Thus  $\rm{Aut_n}(G)$ is finitely generated.
Let us prove now that $\rm{Aut_n}(G)$ is nilpotent-by-(finite and supersoluble).
By a result of Zappa  \cite[5.4.8]{RO1}, there is a normal series
\begin{equation}
1=G_m\unlhd G_{m-1}\unlhd\cdots\unlhd G_1\unlhd G_0=G
\label{series}
\end{equation}
in which each factor is cyclic of prime or infinite order. For any
$k\in\{0,1,\ldots,m\}$, we denote by $\Gamma_k$ the set of all normal automorphisms 
of $G$ which stabilize the series
$G_k\unlhd G_{k-1}\unlhd\cdots\unlhd G_1\unlhd G_0=G$ (we put
$\Gamma_0=\rm{Aut_n}(G)$). 
Clearly,  $\Gamma_0,\Gamma_1,\ldots,\Gamma_m$ forms a decreasing  
sequence of normal subgroups of $\rm{Aut_n}(G)$.
Using once again the result of Kalu\v{z}nin \cite[p. 9]{SE}, we can assert
that $\Gamma_m$ is nilpotent (of class at most $m-1$) since $\Gamma_m$
stabilizes the series (\ref{series}) above.
It remains to prove that $\rm{Aut_n}(G)/\Gamma_m$ is supersoluble and finite.
For each integer $k$ (with $0\leq k\leq m-1$), consider
the homomorphism $\Psi_k\colon \Gamma_k\to \rm{Aut}(G_k/G_{k+1})$,
where for any $f\in  \Gamma_k$,
$\Psi_k(f)$ is defined as the automorphism induced by $f$ on $G_k/G_{k+1}$.
We observe that $\rm{Aut}(G_k/G_{k+1})$ is finite cyclic and that 
$\ker\Psi_k=\Gamma_{k+1}$. Consequently,  the factor $\Gamma_{k}/\Gamma_{k+1}$
is finite cyclic, since it is isomorphic to a subgroup of $\rm{Aut}(G_k/G_{k+1})$.
It follows that
$$1=\left(\Gamma_{m}/\Gamma_{m} \right)
\unlhd\left(\Gamma_{m-1}/\Gamma_{m} \right)\unlhd
\cdots\unlhd \left(\Gamma_{1}/\Gamma_{m} \right)
\unlhd \left(\Gamma_{0}/\Gamma_{m} \right)=\left(\rm{Aut_n}(G)/\Gamma_{m} \right)$$
forms a normal series in which each factor is finite cyclic. Thus 
$\rm{Aut_n}(G)/\Gamma_m$ is supersoluble and finite and the proof is complete. 
\end{proof} 
%
%

%
\end{document}